\documentclass%[12pt
[reqno]{amsart}
\usepackage[utf8]{inputenc}
\usepackage{amssymb,amscd}
\usepackage{thmtools}

\usepackage[textsize=tiny]{todonotes}
\usepackage{hyperref}
\usepackage[shortlabels]{enumitem}
\usepackage{float} 
\usepackage{cleveref}

\newtheorem{definition}{Definition}

\newtheoremstyle{alstandard}{7pt}{3pt}{\rm}{}{\scshape}{:}{0.5em}{}
\theoremstyle{alstandard}
\swapnumbers
\declaretheorem[name=Theorem]{theorem}
\numberwithin{theorem}{section}
\declaretheorem[sibling=theorem, name=Lemma]{lemma}
\declaretheorem[sibling=theorem, name=Proposition]{prop}

\declaretheorem[sibling=theorem, name=Definition]{defi}

\declaretheorem[sibling=theorem, name=Corollary]{cor}
\declaretheorem[sibling=theorem, name=Remark]{remark}
\declaretheorem[sibling=theorem, name=Notation]{notation}

\declaretheorem[sibling=theorem, name=Conjecture]{conjecture}
\newcommand{\C}{\mathbb{C}}
\newcommand{\N}{\mathbb{N}}

\renewcommand\P{\mathbb P}

\newcommand{\dual}{\vee}

\DeclareMathOperator{\rk}{rank}

\DeclareMathOperator{\GM}{GM}

\DeclareMathOperator{\grass}{Gr}
\DeclareMathOperator{\chow}{Chow}
\DeclareMathOperator{\segreveronese}{s\nu}
\DeclareMathOperator{\apex}{ap}
\DeclareMathOperator{\smooth}{reg}

\DeclareMathOperator{\GL}{GL}
\DeclareMathOperator{\SO}{SO}
\DeclareMathOperator{\Sp}{Sp}

\usepackage{pifont}
\begin{document}
	\title[Invariant Secant Varieties]{Nondefectivity of invariant secant varieties} 
	
	\author[Blomenhofer]{Alexander Taveira Blomenhofer}
	\address{Centrum Wiskunde en Informatica (CWI), Science Park 123, 
		1098XG Amsterdam, The Netherlands}
	\email{atb@cwi.nl}
	
	\author[Casarotti]{Alex Casarotti}
	\address{Università di Bologna, Piazza di Porta S. Donato 5, (40126) Bologna, Italy}
	\email{alex.casarotti@unibo.it}

	\begin{abstract} 
		We show that a large class of secant varieties is nondefective. In particular, we positively resolve most cases of the Baur-Draisma-de Graaf conjecture on Grassmannian secants in large dimensions. Our result improves the known bounds on nondefectivity for various other secant varieties, including Chow varieties, Segre-Veronese varieties and Gaussian moment varieties. We also give bounds for identifiability and the generic ranks. 
	\end{abstract}

	%\thanks{ }
	\maketitle

	%proposed division in sections and respective tex input files
	\raggedbottom
\section{Introduction}\label{sec:intro}

Our starting point is the empirical observation that many secant varieties behave highly ``regular'' with respect to dimension. If one picks a ``reasonable'' base variety $ V $ of small dimension $ N $ in a space $ S^d(\C^n) $ of polynomials, then very often, the first secant varieties $ \sigma_m(V) $ will have the \emph{expected dimension} for small $ m $, which, in affine notation, equals $ mN $. 

Conversely, for large $ m $, the secants of a nondegenerate variety $ V $ will fill up the entire space $ S^d(\C^n) $. The smallest $ m $ such that the $ m $-th secant fills the whole space is called the \emph{generic $ V $-rank}, denoted $ m^{\circ} $. Frequently, the lower bound
\begin{align}
	m^{\circ}\ge \frac{\dim S^d(\C^n)}{\dim V}, 
\end{align}
obtained from counting parameters, is a very close estimate for the true value of $ m^{\circ} $. One of the first proven results of this type was the Alexander-Hirschowitz theorem \cite{hirschowitz1995polynomial}, which concerned decompositions
\begin{align}\label{eq:1-Waring-dec}
	f = \sum_{i = 1}^{m} \ell_i^d
\end{align}
of a $ d $-form $ f $ in $ n $ variables as a sum of $ d $-th powers of linear forms $ \ell_{1},\ldots,\ell_{m} $. Decompositions like \Cref{eq:1-Waring-dec} are called $ 1 $-Waring decompositions. The value $ m $ is called the $ 1 $-Waring \emph{rank} of the decomposition. In this example, $  V = \nu_d(\C^n) $ is the degree-$ d $ Veronese embedding of $ \C^n $. Any form $ f $ with a decomposition like \Cref{eq:1-Waring-dec} lies in the $ m $-th secant variety $ \sigma_m (V) $. The Alexander-Hirschowitz theorem states that for all but finitely many $ (n, d) $ with $ n\ge 2, d\ge 3 $, the generic $1$-Waring rank is the rounded-up value of $ \frac{\dim S^d(\C^n)}{\dim V} $, and that most secants of rank smaller than $ m^{\circ} $ have the expected dimension. 

If the linear forms $ \ell_i $ are replaced by forms $ q_i $ of higher degree $ k $, then one obtains the notion of a $ k $-Waring decomposition 
\begin{align}\label{eq:k-Waring-dec}
	f = \sum_{i = 1}^{m} q_i^d
\end{align}
of a form $ f $ of degree $ kd $. Those higher-order Waring decompositions can be traced back a long time. E.g., Ramanujan \cite{ramanujan1913problem} posed a famous problem about sums of cubes of quadratics in 1913. In recent times, the study of Waring decomposition was fueled by connections to Machine Learning \cite{Bafna_Hsieh_Kothari_Xu_2022} and circuit complexity \cite{Garg_Kayal_Saha_2020}. % and polynomial optimization. 
Defectivity and identifiability properties with respect to $k$-Waring problem have been treated in  \cite{Casarotti_Postinghel_2023}.

Yet another type of decompositions within $ S^{d}(\C^n) $ are Chow decompositions. Here, one aims to write a $ d $-form 
\begin{align}
	f = \sum_{i = 1}^{m} \ell_{i, 1}\cdot \ldots \cdot \ell_{i, d}
\end{align}
as a sum of $ d $-fold products of linear forms $ \ell_{i, j} $. All of these decompositions have one thing in common: They start with an irreducible, affine cone $ V \subseteq S^{d}(\C^n) $ that is closed under the action of $ \GL(\C^n) $. Subsequently, they try to write an element $ f $ of $ S^{d}(\C^n) $ as a sum of finitely many elements of $ V $.

Therefore, one may produce more problems of this type, by continuing the list of $ \GL $-invariant subvarieties of $ S^{d}(\C^n) $. Machine Learning motivates to look at \emph{Gaussian moment varieties}, since their secants are important to Gaussian mixture models.

A very similar behaviour is encountered, if instead of the symmetric powers, one looks at canonical subvarieties of the alternating powers $ \bigwedge^{d}(\C^n) $. The most natural choice is the Grassmannian variety $ \grass(n, d) $ of $ d $-dimensional subspaces of $ \C^n $. It may be embedded into $ \P(\bigwedge^{d}(\C^n)) $ via the \emph{Plücker map}
\begin{align}
	\langle u_1,\ldots,u_d \rangle \mapsto [u_1 \wedge \ldots \wedge u_d],
\end{align}  
which sends a basis $ u_1,\ldots,u_d $ of the space to the projective equivalence class of its wedge product. Secant varieties of Grassmannians are of significance to coding theory, as outlined by Baur, Draisma and de Graaf \cite{Baur_Draisma_DeGraaf_2007}. 

\subsection{Contributions} We give a unified treatment and bounds for nondefectivity and the generic rank in all the cases mentioned above. For Grassmannians, Chow varieties, Gaussian moment and Segre-Veronese varieties, our new bounds improve significantly on the best known ones and are asymptotically optimal for large $ n $. The common structure in the examples above is the presence of a group action, e.g., $ G = \GL(\C^n) $, and of an irreducible $ G $-module $ \mathcal{L} $, e.g., $ \mathcal{L} = S^d(\C^n) $ or $ \mathcal{L} = \bigwedge^{d}(\C^n) $. We leverage invariant-theoretic techniques to show that all of the above-mentioned varieties, and many more, will have most of their secant dimensions as expected. 
Our main result is \Cref{thm:stationarity-lemma}, but the applications all follow from its Corollary,  \Cref{thm:introduction-main-result}, which we will state here and prove in \Cref{cor:generalized-nenashev-nondefectivity}. 
\begin{theorem}
	\label{thm:introduction-main-result}
	Let $ G $ be a group and $ V $ an $ N $-dimensional, irreducible affine cone in an irreducible $ G $-module $ \mathcal{L} $. If $ V $ is closed under the $ G $-action, then it holds that 
	\begin{enumerate}
		\item $ V $ is $ m $-nondefective for all $  m\leq \frac{\dim \mathcal{L}}{N} - N $.
		\item The generic rank (with respect to $ V $) is between $ \frac{\dim \mathcal{L}}{N} $ and $\frac{\dim \mathcal{L}}{N} + N$. 
	\end{enumerate}
\end{theorem}

Despite its generality, \Cref{thm:introduction-main-result} gives better bounds than quite a few specialized results for specific varieties.   
In \Cref{sec:applications}, we list cases where our new result improves on the currently best-known bounds for nondefectivity, identifiability and/or the generic rank. In many of those cases, our bounds are asymptotically optimal for large $ n $. In particular, we prove most cases of the Baur-Draisma-de Graaf conjecture, stated in \Cref{conj:baur-draisma-degraaf}. 

\subsection{Organization} We prove our main result in \Cref{sec:stationarity-lemma}. In \Cref{sec:applications}, we apply it to a plethora of different varieties, where we can obtain improved bounds on nondefectivity and the generic rank.

\subsection{Techniques and Acknowledgments}

In the 1980's, Ådlandsvik examined varieties, which have a large number of defective secant varieties, before they become (projectively) a cone. In 2015, Nenashev resolved a large number of cases of Fröberg's conjecture for forms of equal degree via a very short, clean and elegant argument. Both independently made observations, which are special cases of our main technical result, \Cref{lem:stationarity-lemma}. In this work, we present a generalization of both.

We are grateful to Edoardo Ballico for pointing us towards the work of Bjørn Ådlandsvik, and to Vincenzo Galgano, who suggested the application to spinor varieties. We also wish to thank Alessandro Oneto for inviting the authors to the University of Trento and Nick Vannieuwenhoven for telling us about the (non-)weak defectivity of Chow varieties.

\section{Preliminaries}

Let us briefly recall the main notions used in this paper. Most of our varieties will be affine cones $ V $ in a linear space $ \mathcal{L} $, i.e. $\mathbb{C} \cdot V \subseteq V$. The corresponding projective varieties we denote by $ \P(V) \subseteq \P(\mathcal{L}) $. Frequently, we will also have a group $ G $ acting on the space $ \mathcal{L} $, giving it the structure of a $ G $-module. 
We say that a variety $ V $ is $ G $-\emph{invariant}, if for all $ v\in V $ and $ g\in G $, it holds $ gv\in V $. We use two different notions of irreducibility, that are not to be confused: A variety is called irreducible, if it cannot be covered by the union of two proper subvarieties. A $ G $-module $ \mathcal{L} $ is called irreducible, if $ \mathcal{L} $ is not the zero module and $ \mathcal{L} $ does not contain any proper nonzero $ G $-submodules. The set of smooth points of a variety $ V $ is denoted $ V_{\smooth} $. For every smooth point $x \in V$ we denote by $T_x V$ the tangent space to $V$ at $x$. If $x_1,\dots,x_m$ are smooth points of $V$ we denote by $\langle T_{x_1}V_1,\ldots, T_{x_m} V_m  \rangle$ the linear span of the corresponding tangent spaces.  The space of $ d $-forms on a vector space $ \mathcal{L} $ is denoted $ S^d(\mathcal{L}) $. A special case is the dual space $ \mathcal{L}^{\dual} = S^1(\mathcal{L}) $.

\begin{definition}
	The \emph{$m$-th secant variety} $\sigma_m V$ of the affine cone $V$ is the Zariski closure of 
	\begin{align}
		\{x_1 + \ldots + x_m \mid x_1,\ldots,x_m \in V\}. 
	\end{align}
\end{definition}
The $m$-th secant variety has the \emph{expected dimension} $e_m(V):=\min\{mN, \dim \mathcal{L}\}$. If the dimension of $ \sigma_m V $ is smaller than expected, we say that $V$ is $m$-\emph{defective}. Otherwise, $ V $ is called $ m $-\emph{nondefective}. We say that $V$ is $m$-\textit{identifiable}, if the general point $ p\in \sigma_m V$ has, up to permutation, a unique representation $ p = x_1 + \ldots + x_m $, where $ x_1,\ldots,x_m \in V $. 
Secant varieties are a special case of joins: The \emph{join} of affine cones $ V_1,\ldots,V_m $ is defined as the closure of $ \{x_1 + \ldots + x_m \mid x_1\in V_1,\ldots,x_m \in V_m\} $. We denote it by $ J(V_1,\ldots,V_m) $. 

The main tool in order to compute the dimension of joins and secant varieties dates back to Terracini \cite{Te12} and can be summarized as follows:

\begin{lemma}[{Terracini, see \cite{Te12}}]
	For general $ x_1 \in V_1,\ldots,x_m \in V_m $ and general $ z\in \langle x_1,\ldots,x_m \rangle $, the tangent space to the join $ J(V_1,\ldots,V_m) $ at  $ z $ equals $\langle T_{x_1}V_1,\ldots, T_{x_m} V_m  \rangle$. 
	
	Moreover, for smooth points $ x_1 \in V_1,\ldots,x_m \in V_m $, and any $ z\in \langle x_1,\ldots,x_m \rangle $, $\langle T_{x_1}V,\ldots, T_{x_m} V  \rangle$ is still contained in the tangent to $ J(V_1,\ldots,V_m)$ at $ z $. 
\end{lemma}

\begin{prop}\label{prop:apex-space}
	The \emph{apex space} of an (irreducible) affine cone $ V $ is defined as
	\begin{align}
		\apex(V) = \{x\in V \mid V+x = V\}.
	\end{align}
%	If $ V $ is an irreducible, affine cone, then, $ \apex(V) = \bigcap_{x \in V_{\smooth}} T_{x}V $. 
	For brevity, we also denote $ \apex_m (V) := \apex(\sigma_mV) $. It holds that 
	\begin{align}
		 \apex_m(V) = \bigcap_{(x_1,\ldots,x_m) \in \mathcal{U}_{m, V} } \langle T_{x_1} V,\ldots, T_{x_m} V \rangle, 
	\end{align}
	where $ \mathcal{U}_{V, m} $ denotes the set 
	\begin{align}
		 \{x\in V_{\mathrm{reg}}^m \mid \langle T_{x_1} V,\ldots, T_{x_m} V \rangle = T_{x_1 + \ldots + x_m} \sigma_m V \text{ has generic dimension} \}.
	\end{align}
\end{prop}
\begin{proof}
	For the inclusion from left to right, let $ p\in \apex_m(V) $ and let $ (x_1,\ldots,x_m) \in \mathcal{U}_{V, m} $. By definition of the $ m $-apex, it holds $ J(\sigma_m V, \langle p \rangle) = \sigma_m V$. Terracini's Lemma applied to the join yields $ \langle T_{x_1} V,\ldots, T_{x_m} V, p \rangle \subseteq \langle T_{x_1} V,\ldots, T_{x_m} V\rangle$. Thus, $ p $ lies in $ \langle T_{x_1} V,\ldots,T_{x_m} V  \rangle $. 
	
	Conversely, let $ p $ in the right hand side. Choose a general point $ x\in \mathcal{U}_{V, m} $. Again by Terracini's Lemma, equality $ \langle T_{x_1} V,\ldots, T_{x_m} V, p \rangle = \langle T_{x_1} V,\ldots, T_{x_m} V\rangle $ holds. Since $ x $ is general, thus the dimensions of $ J(\sigma_m V, \langle p \rangle) $ and of $ \sigma_m V $ are equal. By irreducibility, the varieties are equal. Thus $ p\in \apex_m (V) $.
\end{proof}

\begin{remark}
	The apex space is a linear space satisfying $ \apex(V) \subseteq V $. 
	If $ \mathcal{L} $ is a $ G $-module and $ V $ is $ G $-invariant, then $ \apex(V) $ is a $ G $-submodule of $ \mathcal{L} $. 
\end{remark}

\section{The stationarity lemma}\label{sec:stationarity-lemma}

Let $ G $ a group acting on the linear space $ \mathcal{L} $, such that $ \mathcal{L} $ is an irreducible $ G $-module. 
We consider irreducible algebraic varieties $ V, V' $, which are also embedded in some $ G $-modules and which are closed under the $ G $-action. Furthermore, we consider $ G $-equivariant (algebraic) maps 
\begin{align}\label{eq:statement-T-U}
	T\colon \P(V) \to \mathrm{Gr}(N, \mathcal{L}), [x] \mapsto T_{x},\\
	U\colon \P(V') \to \mathrm{Gr}(N', \mathcal{L}),[y] \mapsto U_{y}, 
\end{align}
By $ \grass(N, \mathcal{L}) $, we denote the Grassmannian variety of $ N $-dimensional subspaces of $ \mathcal{L} $. Note that $ G $ acts canonically on the Grassmannian varieties of $ \mathcal{L} $.

Let $ \mathcal{U}_{T, m} =  \{(x_1,\ldots,x_m) \mid \langle T_{x_1},\ldots,T_{x_m} \rangle \text{ is of generic dimension} \}$. In analogy to \Cref{prop:apex-space}, we define the $ m $-th apex of the map $ T $ to be
\begin{align}
	\apex_m(T) :=  \bigcap_{(x_1,\ldots,x_m) \in \mathcal{U}_{T, m}} \langle T_{x_1}, \ldots, T_{x_m} \rangle 
\end{align}
Taking $ T\colon V_{\smooth} \to \grass(\dim V, \mathcal{L}) $ as the tangent map, one recovers the apex space from \Cref{prop:apex-space}.

\begin{lemma}\label{lem:apex-lemma}
	Let $ \mathcal{U} $ a dense open subset of $ \mathcal{U}_{T, m} $. Then 
	\begin{align}
		\apex_m(T) = \bigcap_{(x_1,\ldots,x_m) \in \mathcal{U}} \langle T_{x_1}, \ldots, T_{x_m} \rangle 
	\end{align}
\end{lemma}
\begin{proof}
	The inclusion from left to right is clear. Now, let $ p $ be contained in the right hand side. Then, $ p $ lies in $ \langle T_{x_1}, \ldots, T_{x_m} \rangle  $ for all $ x = (x_1,\ldots,x_m) \in \mathcal{U} $. Let $ R $ denote the generic dimension of $ \langle T_{x_1}, \ldots, T_{x_m} \rangle  $. For all $ x\in \mathcal{U}_{T, m} $, by the Plücker embedding, we can represent $ \langle T_{x_1}, \ldots, T_{x_m} \rangle  $ as a projective equivalence class of an element $ t(x_1,\ldots,x_m) \in \bigwedge^{R}(\C^n) $, where $ t $ depends polynomially on $ (x_1,\ldots,x_m) $. Define $ f(p, x) = p\wedge t(x_1,\ldots,x_m) $. Clearly, $ f(p, x) $ is the zero tensor for all $ x\in \mathcal{U} $. By continuity, it is constantly zero. But for $ x\in \mathcal{U}_{T, m} $, the identity $ p\wedge t(x_1,\ldots,x_m) = 0  $ implies that $ p\in  \langle T_{x_1}, \ldots, T_{x_m} \rangle $.
\end{proof}

%\end{defi}
The next Lemma is a generalization of previous results from \cite{Adlandsvik_1988} and \cite{nenashev2017note}. For any $ m\in \N $, define:
\begin{align}\label{eq:statement-T-cap-U}
	a_m:=\dim(\langle T_{x_1} ,\ldots,T_{x_m} \rangle \cap U_{y}) 
\end{align}
where $ x_1,\ldots,x_m \in V, y\in V' $ are generic points. Clearly, we have that 
\begin{align}
	0 = a_0 \leq a_1 \leq a_2 \leq \ldots
\end{align}

\begin{lemma}[Stationarity Lemma]\label{lem:stationarity-lemma}
	Let $ \mathcal{L} $ an irreducible $ G $-module and $ V $, $  V' $, $ T $, $  U $, $ (a_m)_{m} $ as introduced above. If $ m \in \N$ is such that $ a_{m+1} = a_{m} \ne 0 $, then $ \langle T_{x_1},\ldots,T_{x_m} \rangle = \mathcal{L}$ and $ a_m = N' $. 
\end{lemma}
\begin{proof}
	Consider generic points $x_1, \ldots, x_m, x_1', \ldots, x_m' \in V$, and generic $y \in V'$. We know that
	\begin{align}
		\dim \langle T_{x_1}, \ldots, T_{x_m} \rangle \cap U_{y} = a_m = a_{m+1} = \dim \langle T_{x_1}, \ldots, T_{x_{m}}, T_{x_{m+1}} \rangle \cap U_{y}. \nonumber
	\end{align}
	Therefore, it also holds
	\begin{align}
	\dim \langle T_{x_1}, \ldots, T_{x_m} \rangle \cap U_{y} = \dim \langle T_{x_1}, \ldots, T_{x_m}, T_{x_m'} \rangle \cap  U_{y}.
	\end{align}
	The space on the left-hand side is a subspace of the space on the right-hand side, of same dimension. Hence, they are equal: 
	\begin{align}
		\langle T_{x_1} , \ldots, T_{x_m} \rangle  \cap U_{y} = \langle T_{x_1}, \ldots, T_{x_m}, T_{x_m'} \rangle \cap  U_{y}.
	\end{align}
	We may apply the same argument with exchanged roles of $ x_m $ and $ x_m' $ to obtain 
	\[
	\langle T_{x_1} , \ldots, T_{x_{m-1}}, T_{x_m'} \rangle  \cap U_{y} = \langle T_{x_1}, \ldots, T_{x_m}, T_{x_m'} \rangle \cap  U_{y}.
	\]
	Together, this implies
	\begin{align}\label{eq:ts-xm-exchange}
		\langle T_{x_1}, \ldots, T_{x_{m-1}}, T_{x_m'} \rangle  \cap U_{y} = \langle T_{x_1}, \ldots, T_{x_m} \rangle \cap U_{y}
	\end{align}
	We thus showed that swapping $ x_m $ with $ x_m' $ does not change the space from \Cref{eq:ts-xm-exchange}. 
	Repeating the procedure with all pairs of $ x_i, x_i' $, we obtain
	\begin{align}\label{eq:ts-exchange-all}
		\langle T_{x_1'} , \ldots, T_{x_m'} \rangle  \cap U_{y} = \langle T_{x_1}, \ldots, T_{x_m} \rangle \cap U_{y}
	\end{align}
	In other words, any point $ p $ in the set from \Cref{eq:ts-exchange-all} lies in $ \langle T_{x_1},\ldots,T_{x_m} \rangle $ for all general choices of $ x_1,\ldots,x_m $. 
	Thus,
	\begin{align}
		 \langle T_{x_1}, \ldots, T_{x_m} \rangle \cap  U_{y} = \apex_m(T) \cap U_y,
	\end{align}
	by \Cref{lem:apex-lemma}. 
	In particular, the $ m $-th apex of $ T $ is not the zero space, since $ a_m > 0 $. The space $ \apex_m(T) $ is $ G $-invariant and $ \mathcal{L} $ is an irreducible $ G $-module. Therefore, it follows that $ \mathcal{L} = \apex_m(T) \subseteq \langle T_{x_1}, \ldots, T_{x_m} \rangle $. 
\end{proof}

\begin{remark}\label{rem:adlandsvik-nenashev}
	Similar arguments were present in the work of Ådlandsvik \cite{Adlandsvik_1988} and Nenashev \cite{nenashev2017note}. Ådlandsvik considered the case where $ T = U $ is the tangent map of a variety $ V $, which associates to $ x\in V_{\smooth} $ the tangent space $ T_{x} V $. See \cite[Proposition 2.1(ii)]{Adlandsvik_1988}. He proves a similar stationarity lemma, but he did not endow $ \mathcal{L} $ with a group action. As a result, he arrives at the weaker conclusion that at a stationarity point $ 0\ne a_m = a_{m+1} $, the apex $ \apex_m(T) $ must not be the zero space. 
		
	Independently, Nenashev proved in \cite{nenashev2017note} a result about the Hilbert series of general ideals. It is also a special case of our result, obtained if one takes $ V $ to be any $ \GL $-invariant variety in $ S^{d}(\C^n) $ and $ \mathcal{L} = S^{d+k}(\C^n) $, which is an irreducible $ \GL$-module. Here, the maps $ T = U $ are equal and given by $ T_{x} = (x)_{d+k} $. In other words, a form $ x \in V $ is mapped to some graded component of the principal ideal it generates. 
\end{remark}

\begin{theorem} \label{thm:stationarity-lemma}
	Let $ \mathcal{L} $ be an irreducible $ G $-module, $ V, V' $ irreducible $ G $-invariant affine cones of dimensions $ N $ and $ N' $, respectively. Let $ T\colon \P(V) \to \grass(N, \mathcal{L})$ and $ U\colon \P(V') \to \grass(N', \mathcal{L})$ be $ G $-equivariant maps. 
	Then, there exists an interval $ \mathcal{I}\subseteq \N $ of length at most $ N' $, such that for all $ m\notin \mathcal{I} $, and generic $ x_1,\ldots,x_m \in V$, generic $ y\in V' $, either one of the following holds true:
	\begin{enumerate}
		\item $ \langle T_{x_1}, \ldots, T_{x_m} \rangle \cap U_{y} = \{0\} $ or 
		\item $ \langle T_{x_1}, \ldots, T_{x_m} \rangle = \mathcal{L}. $
	\end{enumerate}
%	Then, there is a set $ \mathcal{I} $ of at most $ N' $ ``excluded'' values of $ m $, such that for all $ m\notin \mathcal{I} $, and generic $ x_1,\ldots,x_m \in V$, generic $ y\in V' $, either one of the following holds true:
%	\begin{enumerate}
%		\item $ \langle T_{x_1}, \ldots, T_{x_m} \rangle \cap U_{y} = \{0\} $ or 
%		\item $ \langle T_{x_1}, \ldots, T_{x_m} \rangle = \mathcal{L}. $
%	\end{enumerate}
%	Furthermore, $ \mathcal{I} $ can be chosen as an interval in $ \N $. 
\end{theorem}
\begin{proof}
	There exists a smallest $ m_1 \in \N $ such that $ \langle T_{x_1}, \ldots, T_{x_{m_{1}}} \rangle = \mathcal{L} $. There also exists a largest $ m_0\in \N_0 $ such that $ \langle T_{x_1}, \ldots, T_{x_{m_0}} \rangle \cap U_{y} = \{0\} $. The claim holds true for all $ m\notin \mathcal{I}:= \{m_0+1,\ldots,m_1\} $. The length of this interval is $ m_1-m_0 $. 	By \Cref{lem:stationarity-lemma}, we know that the dimension of the intersection is strictly increasing between $ m_0 $ and $ m_1 $, and for $ m_1 $, the intersection equals $ U_{y} $. Therefore, %we must have %$ m_1-m_0 \leq \dim U_{y} = N' $.
	\begin{align}\label{eq:m1-m0-bound}
		m_1-m_0 \leq a_{m_1} - a_{m_0} = \dim U_{y} = N'. 
	\end{align}
	This asserts the claim.
\end{proof}

In the special case $ T = U $, it is possible to deduce explicit bounds for the quantities $ m_0 $ and $ m_1 $, which occur in the proof of \Cref{thm:stationarity-lemma}. In the following, we show this for the even more special case, where $ T = U $ is the tangent map of a variety. In this way, we obtain a very general criterion for secant nondefectivity and estimates for the generic rank with respect to a large class of base varieties. 

\begin{cor}%[{Restatement of \Cref{thm:introduction-main-result}}] 
	\label{cor:generalized-nenashev-nondefectivity}
	Let $ V $ be an $ N $-dimensional, irreducible subvariety of the irreducible $ G $-module $ \mathcal{L} $. Assume $ V $ is $ G $-invariant. Then,  
	\begin{enumerate}
		\item $ V $ is $ m $-nondefective for all $  m\leq \frac{\dim \mathcal{L}}{N} - N $.
		\item The generic rank of $ \mathcal{L} $ with respect to $ V $ is at most $ m^{\circ} \leq \frac{\dim\mathcal{L}}{N} + N $. 
	\end{enumerate}
\end{cor}
\begin{proof}
	In \Cref{thm:stationarity-lemma}, take $ T = U $ as the tangent map 
	\begin{align}
		T\colon V_{\smooth} \to \grass(\mathcal{L}, N), x \mapsto T_{x} V,
	\end{align}
	which takes a point $ x \in V $ to its tangent space. We have $ N = N' = \dim V $ in the notation of \Cref{thm:stationarity-lemma}. Let $ m_0 $ and $ m_1 $ be as in \Cref{thm:stationarity-lemma}. 	From counting parameters, we have the obvious bounds $ N(m_0+1) \leq \dim \mathcal{L} $ and $ N m_1 \ge \dim \mathcal{L} $. Using these together with \Cref{eq:m1-m0-bound}, we obtain 
	\begin{align}
			m_0 \ge m_1 - N \ge \frac{\dim \mathcal{L}}{N} - N,
	\end{align}
	and 
	\begin{align}
			m_1 \le m_0 + N \leq \frac{\dim \mathcal{L} - N}{N} + N = \frac{\dim \mathcal{L}}{N} + N -1.
	\end{align}
	 and we obtain that a sum of $ m $ tangent spaces is direct, as long as $  m\leq \frac{\dim  \mathcal{L}}{N} - N $. Similarly, we also obtain from \Cref{thm:stationarity-lemma} that if $  m\geq \frac{\dim \mathcal{L}}{N} + N $, then the $ m $-th secant variety $ \sigma_m V $ must fill everything. 
\end{proof}

We now state an important result by Massarenti and Mella \cite{Massarenti_Mella_2022}, which enables us to link the study of defectivity of varieties $V$ to their identifiability properties.

\begin{theorem}[{Massarenti, Mella, \cite[Theorem 1.5]{Massarenti_Mella_2022}}]\label{thm:massarenti-mella}
Let $V \subseteq \mathcal{L}$ be an irreducible and non-degenerate variety of dimension $N$, $m \geq 1$ an integer and assume that: 
\begin{enumerate}
\item
$mN \leq \dim \mathcal{L}$,
\item
The fibers of the tangent map $ T\colon \P(V_{\smooth}) \to \grass(N, \mathcal{L}) $ are finite. 
\item
$V$ is not $(m+1)$-defective.
\end{enumerate}
Then, $V$ is $m$-identifiable.
\end{theorem}

If a variety $ V $ satisfies Condition (2) of \Cref{thm:massarenti-mella}, we also say that the tangent map of $ V $ is \emph{nondegenerate}.

\raggedbottom
\section{Applications}\label{sec:applications}

\subsection{Grassmannian varieties}

\begin{notation}
	Let $ d, n \in \N $. We denote by $ \grass(d, n) $ the \emph{Grassmannian variety} of $ d $-dimensional subspaces of $ \C^n $. Via the \emph{Plücker embedding}, we may see $ \grass(d, n) $ as a subvariety of the Schur functor $ \bigwedge^{d}(\C^n) $.  
\end{notation}

Secants to Grassmannian varieties are a natural object to study in tensor geometry. Just a few years after the celebrated Alexander-Hirschowitz result, Catalisano, Geramita and Gimigliano examined their nondefectivity \cite{Catalisano_Geramita_Gimigliano_2003}. Galgano \cite{Galgano_Staffolani_2022} examined non-generic identifiability for $ \sigma_2 \grass(d, n) $. Baur, Draisma and de Graaf \cite{Baur_Draisma_DeGraaf_2007} discovered connections between secant nondefectivity of Grassmannians and coding theory. In addition, their work lead to the following conjecture. 

\begin{conjecture}[Baur–Draisma–de Graaf] \label{conj:baur-draisma-degraaf} Let $ n, d, m\in \N $ and $ d\ge 3 $. Then, the secant variety $ \sigma_m \grass (d, n) $ has the expected dimension $\min(m \cdot d(n-d), \binom{n}{d})$, except for the following cases: \\
	\begin{tabular}{lcc}
		& \textbf{Actual Codimension} & \textbf{Expected Codimension} \\
		$\sigma_3 \grass(3, 7) $ & 1 & 0 \\
		$\sigma_3 \grass(4, 8) $ & 20 & 19 \\
		$\sigma_4 \grass(4, 8) $ & 6 & 2 \\
		$\sigma_4 \grass(3, 9) $ & 10 & 8 \\
	\end{tabular}
\end{conjecture}

The table in \Cref{conj:baur-draisma-degraaf} was taken from \cite[Conj 4.2]{hitchhiker_Bernardi_Carlini_Catalisano_Gimigliano_Oneto_2018}. Note that we adapted to affine notation, so in our table, both $ d $ and $ n $ are each larger by $ 1 $ compared to the table in \cite[Conj 4.2]{hitchhiker_Bernardi_Carlini_Catalisano_Gimigliano_Oneto_2018}. If it was known to be true, \Cref{conj:baur-draisma-degraaf} could be seen as an ``alternating'' version of the Alexander-Hirschowitz theorem, where the Schur functor $ S^d $ of symmetric powers is replaced by the Schur functor $ \bigwedge^{d} $ of alternating powers, and the Veronese embedding is replaced by the Plücker embedding. Fortunately, with \Cref{cor:generalized-nenashev-nondefectivity}, we are able to resolve a lot of cases of the Baur-Draisma-de Graaf conjecture.

\begin{theorem}\label{thm:grassmannian-nondefective}
	For any $ d, n, m\in \N $, the Grassmannian variety $ \grass(d, n) $ is $ m $-nondefective (and $ (m-1) $-identifiable), if 
	\begin{align}\label{eq:grassmannian-nondefective}
		m &\leq \frac{\dim \bigwedge^{d}(\C^n)}{\dim \grass(d, n)} - \dim \grass(d, n)\\
		&=	\frac{1}{d(n-d)}\binom{n}{d} - d(n-d)
	\end{align}
	Conversely, the generic rank of $ \bigwedge^{d}(\C^n) $ with respect to $ \grass(d, n) $ is at most
	\begin{align}
		\frac{1}{d(n-d)}\binom{n}{d} + d(n-d). 
	\end{align}
\end{theorem}
\begin{proof}
	Since the Grassmannian is a smooth variety, its tangent map is nondegenerate. Therefore, $ m $-nondefectivity implies $ (m-1) $-identifiability, by \Cref{thm:massarenti-mella}. The rest follows from \Cref{cor:generalized-nenashev-nondefectivity}.
\end{proof}

\begin{remark} 
	Our \Cref{thm:grassmannian-nondefective} settles a large number of cases the Baur–Draisma–de Graaf conjecture stated below. In fact, for any $ d \ge 3 $, the number of ranks $ m $ for which the conjecture is not settled by our \Cref{thm:grassmannian-nondefective} is asymptotically negligible as a function in $ n $, for large $ n $. For $ d\leq 2 $, our theorem is vacuous. This is also necessary, since secants to Grassmannians of two-dimensional subspaces are known to be non-identifiable. 
	
	Let us compare our \Cref{thm:grassmannian-nondefective} with earlier results. Catalisano, Geramita and Gimigliano \cite{Catalisano_Geramita_Gimigliano_2003} showed nondefectivity for all $ m \leq n/d $, if $ d\ge 3 $. Rischter \cite[Theorem 3.5.1]{Rischter_thesis_2017} showed nondefectivity for all $ m $ bounded by some complicated function in $ n $ and $ d $, with asymptotic growth around $ \mathcal{O}(n^{\log_2(d)+1}) $. Both results are collected in the ``hitchhiker guide'' of Bernardi et. al., see \cite[Theorem 4.3]{hitchhiker_Bernardi_Carlini_Catalisano_Gimigliano_Oneto_2018} and \cite[Theorem 4.5]{hitchhiker_Bernardi_Carlini_Catalisano_Gimigliano_Oneto_2018}.
	For large $ n $, our bound is of course strictly better than both previous results, and it matches the asymptotics of the Baur-Draisma-de Graaf conjecture. \Cref{fig:grassmann-bound-comparison} shows that it already overtakes the bounds from \cite{Rischter_thesis_2017} and \cite{Catalisano_Geramita_Gimigliano_2003} in the case $d=4$ at $ n = 20 $. 
\end{remark}

\begin{figure}
	\centering
\includegraphics[width=0.7\linewidth]{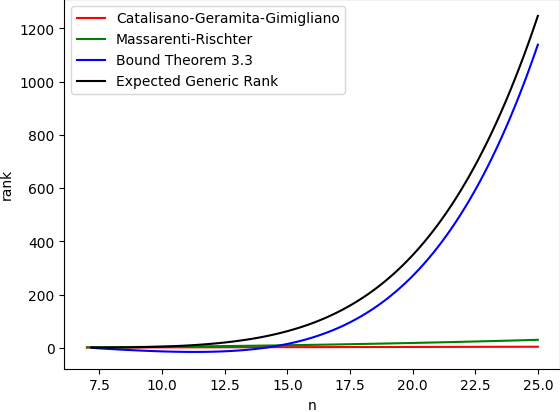}
	\caption{Comparison of the various bounds on secant nondefectivity of $\grass(d,n)$ when $d=5$.}
	\label{fig:grassmann-bound-comparison}
\end{figure}

From \Cref{thm:introduction-main-result}, similar identifiability results may also be derived for flag varieties. 

\begin{notation}
Let $(k_1,\dots,k_r;n)$ denote a sequence of increasing positive numbers of the form $k_1 \leq \dots \leq k_r \leq n$. We denote by $F(k_1,\dots,k_r;n)$ the \emph{Flag variety} parametrizing chains of the forms 
\[
V_1 \subset \dots \subset V_r \subset \mathbb{C}^n
\]
where every $V_i$ is a linear subspace of dimension $\dim(V_i)=k_i$.
\end{notation}

In particular the Flag variety can be seen as a subvariety of the product of Grassmannians $\prod_{i=1}^{r}\grass(k_i,n)$. Restricting the Segre product of the Plücker embedding to $F(k_1,\dots,k_r;n)$ exhibits the Flag as a subvariety of the irreducible $\GL_n^r$-module $\Gamma_a \subseteq \bigwedge^{k_1}\mathbb{C}^n \otimes \ldots \otimes \bigwedge^{k_r}\mathbb{C}^n$, see for instance \cite{Casarotti_Freire_Massarenti}.
Note that, if we define $k_0:=0$, we have 
\[
\dim F(k_1,\dots,k_r;n)=\sum_{j=1}^{r}(n-k_j)(k_j-k_{j-1})
\]
To the best of our knowledge the best bound for secant defectivity and identifiability for Flag varieties is given in \cite[Theorem 4.11]{Casarotti_Freire_Massarenti}, where the authors proved that under the assumption $n \geq 2k_r$ the Flag variety $F(k_1,\dots,k_r;n)$ is not $m$-defective for 
\begin{align}
m \leq \left(\frac{n}{k_r}\right)^{\lfloor \log_2(\sum k_j-1)\rfloor}
\end{align}
Using our result from \Cref{thm:introduction-main-result} we can give the following improvement:

\begin{theorem}
The Flag variety $F(k_1,\dots,k_r;n)$ is $m$-nondefective  (and $ (m-1) $-identifiable), if $m$ is bounded by:
\begin{align}
\frac{\prod_{i=1}^{r}{n \choose k_i}}{\sum_{j=1}^{r}(n-k_j)(k_j-k_{j-1})}-\sum_{j=1}^{r}(n-k_j)(k_j-k_{j-1}) 
\end{align}
Furthermore, the generic rank $m^{\circ}$ is at most
\begin{align}
\frac{\prod_{i=1}^{r}{n \choose k_i}}{\sum_{j=1}^{r}(n-k_j)(k_j-k_{j-1})}+\sum_{j=1}^{r}(n-k_j)(k_j-k_{j-1}) 
\end{align}

\end{theorem}

\subsection{Chow varieties}

In recent work by Torrance and Vannieuwenhoven, see \cite{Torrance_Vannieuwenhoven_2020}, \cite{Torrance_Vannieuwenhoven_2022}, it was shown that almost all secants of the Chow variety
\begin{align}
	\chow_{d}(\C^n) = \{\ell_{1}\cdot \ldots \cdot \ell_{d} \mid \ell_{1},\ldots,\ell_{d} \in S^1(\C^n) \}
\end{align}
are identifiable, in the case $ d = 3 $ of cubics. We can show a result for all degrees $ d\ge 3 $. 

\begin{theorem}
	The Chow variety $ \chow_{d}(\C^n) $  is $ m $-nondefective (and $ (m-1) $-identifiable), if $ m $ is at most $\frac{1}{d(n-1) + 1}\binom{n+d-1}{d} - d(n-1) - 1  $.
	Its generic rank is at most $ \frac{1}{d(n-1) + 1}\binom{n+d-1}{d} + d(n-1) + 1 $.
\end{theorem} 
\begin{proof}
	The Chow variety is $ \GL $-invariant and lives in the Schur functor $ S^d(\C^n) $. The statement thus follows from \Cref{thm:introduction-main-result}, after plugging in the dimension of $ \chow_{d}(\C^n)$, which is $ (n-1)d + 1 $, and of $ S^d(\C^n) $. Oeding computed the dimension of the dual of the Chow variety in \cite{Oeding_2011}. As Torrance and Vannieuwenhoven pointed out in  \cite[Section 2.1]{Torrance_Vannieuwenhoven_2022}, Oeding's result implies that $ \chow_{d}(\C^n) $ is not $ 1 $-weakly defective for $ d\ge 3 $ and $ n\ge 2 $. Hence, the Chow variety is also not $ 1 $-tangentially weakly defective and \Cref{thm:massarenti-mella} implies $ (m-1) $-identifiability. 
\end{proof}

\subsection{Segre-Veronese varieties}

Given $ t\in \N $, vector spaces $ \underline{L} = (L_1,\ldots,L_{t})  $ of dimensions $ \underline{n} = (n_1,\ldots,n_t) $ and nonnegative integers $  \underline{d} = (d_1,\ldots,d_t) $, the Segre-Veronese variety on $ \underline{L} $ of multidegree $ \underline{d} $ is defined as 
\begin{align}
	\segreveronese_{\underline{d}}(\underline{L}) := \{ \ell_{1}^{d_1} \otimes \ldots \otimes \ell_{t}^{d_t} \mid \ell_{1} \in L_1^{\dual},\ldots,\ell_{t} \in L_t^{\dual} \}.
\end{align}
It lives in $ S^{d_1} (L_1) \otimes \ldots \otimes S^{d_t} (L_t) $. This space is naturally endowed with the action of $ G := \GL(L_1)\times \ldots \times \GL(L_t) $.

\begin{theorem}\label{our segre-veronese}
	The Segre-Veronese variety of multidegree $ \underline{d} $ and dimensions $ \underline{n} $ is $ m $-nondefective (and $ (m-1) $-identifiable), if $ m $ is bounded by
	\begin{align}
		\frac{1}{ n - t + 1}\binom{n_1 + d_1 - 1}{d_1} \cdots \binom{n_t + d_t - 1}{d_t} - (n - t + 1),
	\end{align}
	where $ n := n_1 + \ldots + n_t $. Furthermore, the generic rank $ m^{\circ} $ is at most 
	\begin{align}
		\frac{1}{n - t + 1}\binom{n_1 + d_1 - 1}{d_1} \cdots \binom{n_t + d_t - 1}{d_t} + (n - t + 1).
	\end{align}
\end{theorem} 
\begin{proof}
	As outlined before, $ \mathcal{L} :=  S^{d_1} (L_1) \otimes \ldots \otimes S^{d_t} (L_t) $ has the structure of a $ G $-module, where $ G := \GL(L_1)\times \ldots \times \GL(L_t) $. But $ \mathcal{L} $ is in fact an irreducible $ G $-module, since it is the tensor product of the irreducible $ \GL(L_i) $-modules $ S^{d_i}(L_i) $, for $ i \in  \{1,\ldots,t\} $.  
	Note that $ \segreveronese_{\underline{d}}(\underline{L}) $ is invariant with respect of the action of $ G $, and of dimension $ n_1 + \ldots + n_t - t + 1$. Therefore, \Cref{thm:introduction-main-result} is applicable. Calculating the dimensions of $ \mathcal{L} $ and the Segre-Veronese variety explicitly yields the result on nondefectivity. Since Segre-Veronese varieties are smooth, \Cref{thm:massarenti-mella} yields $ (m-1) $-identifiability.
\end{proof}

To the best of our knowledge the actual best bound in the literature is due to Laface-Massarenti-Rischter in \cite[Theorem 1.1]{LaMaRi2022}. They give an asymptotical sharp bound on $m$ of the form:

\begin{align}
\frac{d_j}{(n_j+d_j-1)\cdot (n - t + 1)} \binom{n_1 + d_1 - 1}{d_1} \cdots \binom{n_t + d_t - 1}{d_t}
\end{align}
where 
\[
\frac{n_j-1}{d_j}=\max_{1 \leq i \leq t}\{\frac{n_i-1}{d_i}\}.
\]

Our bound from Theorem \ref{our segre-veronese} however covers asymptotically more cases then Laface-Massarenti-Rischter one. In Figure \ref{fig:segre-veronese-bound-comparison} we make the comparison between the two bounds in the case of $3$ factors Segre-Veronese of fixed degree $\underline{d}=(1,1,2)$ and fixed dimension $n$.

\begin{figure}
	\centering
\includegraphics[width=0.7\linewidth]{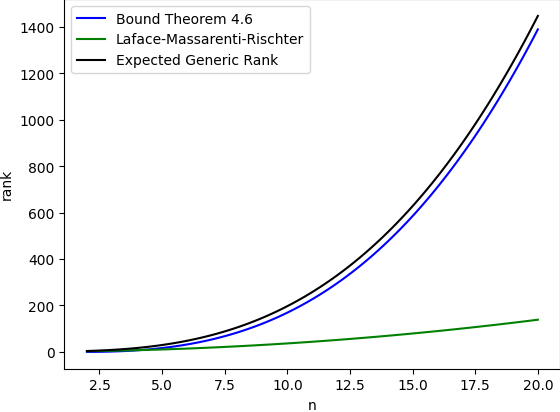}
	\caption{Comparison of the various bounds on secant nondefectivity of $s\nu_{\underline{d}}(\underline{n})$ when $\underline{d}=(1,1,2)$ and $\underline{n}=(n,n,n)$.}
	\label{fig:segre-veronese-bound-comparison}
\end{figure}

\subsection{The Gaussian moment variety}\label{sec:gm-ident}
Gaussian mixtures are an ubiquitous model in Machine Learning. Their identifiability is a long-standing open problem, with the univariate case dating back to Pearson, see \cite{Pearson_1900}. We refer to \cite{Amendola_Faugere_Sturmfels_2016} and \cite{Amendola_Ranestad_Sturmfels_2018} for some historical remarks. They are closely connected to secants of the \emph{Gaussian moment variety}, defined as follows. 

\begin{defi}\label{def:gm-variety}
	Let $ n, d\in \N $. We call the closure of the image of 
	\begin{align}
			s_{d}\colon  S^1(\C^n) \times S^2(\C^n) \to S^{d}(\C^n), (\ell, q)\mapsto \exp\left(\ell + \frac{q}{2}\right)_{=d},
		\end{align}
	the degree-$ d $ \emph{Gaussian moment variety} $ \GM_{d}(\C^n) $ on $ \C^n $. Here, $ \exp $ is to be understood as a formal power series on $ \C^n $ and the subscript denotes its $ d $-homogeneous part. 
\end{defi}

The parameter $ \ell $ is a linear form, which corresponds to the mean of the Gaussian, while the parameter $ q $ is a quadratic form, corresponding to the covariance matrix. 
In \cite{Blomenhofer_Casarotti_Michalek_Oneto_2022} and \cite{Taveira_Blomenhofer_GM6_2023}, the authors gave thorough motivations of the Gaussian moment variety and explained the connection between the statistical estimation problem for Gaussian mixtures and the algebraic identifiability problem for the Gaussian moment variety. In particular, they also showed why for generic identifiability, the constrains of realness and positive-definiteness of $ \ell $ and $ q $, respectively, may be ignored. We settle most cases of Gaussian identifiability for high-dimensional models in the degrees $ d = 5,\ldots,9 $. Note that in applications, the lower degrees are the most relevant ones. Degree $ 5 $ is minimal for any nontrivial identifiability result.

\begin{theorem}\label{thm:gm-nondefective}
	For $ d\ge 5 $ and $ n\ge 2 $, the following hold: 
	\begin{enumerate}
		\item The Gaussian moment varieties $ \mathrm{GM}_d(\C^n) $ are $ m $-nondefective, as long as 
		\begin{align*}
			m \leq \frac{\dim S^d(\C^n)}{\dim \mathrm{GM}_d(\C^n) } - \dim \mathrm{GM}_d(\C^n) = \frac{2{n+d-1 \choose d}}{n (n + 3)} - \frac{1}{2} n (n + 3). 
		\end{align*}  
		\item In addition, the generic rank with respect to the Gaussian moment variety is at most 
		\begin{align*}
			\frac{{n+d-1 \choose d}}{\dim \mathrm{GM}_d(\C^n) } + \dim \mathrm{GM}_d(\C^n).
		\end{align*}
		\item For $ d = 5,\ldots,9 $, $ m $-nondefectivity of $ \mathrm{GM}_d(\C^n) $ implies $ (m-1) $-identifiability. 
	\end{enumerate}
\end{theorem}
\begin{proof}
	The first two claims are a direct consequence of \Cref{thm:introduction-main-result}. Note that the dimension of $ \GM_{d}(\C^n) $ is $ n + \binom{n+1}{2} $, since it is parameterized by $ S^1(\C^n) \oplus S^2(\C^n) $. This value can also be written as $ \frac{1}{2} n (n+3) $. For the third claim, we use \Cref{thm:massarenti-mella}, which shows that $ m $-nondefectivity implies $ (m-1) $-identifiability, if the tangent map is nondegenerate. The tangent map condition was verified for $ d = 5,\ldots,9 $ with a partially computer-assisted proof in \cite[Proposition 2.13]{Taveira_Blomenhofer_GM6_2023}.  
\end{proof}

\begin{remark}\label{rem:asymptotic-optimality-gaussians}
	For $ d\leq 4 $, the Gaussian moment varieties are already $ 2 $-defective, cf. for instance \cite{Taveira_Blomenhofer_GM6_2023}. On the other hand, for $ d\geq 5 $, the dimension of $ \mathrm{GM}_d(\C^n) $ grows of smaller order in $ n $ than the parameter counting rank. Therefore, \Cref{thm:gm-nondefective} asymptotically covers almost all interesting cases of nondefectivity, cf. \Cref{rem:asymptotic-optimality}. We conjecture that $ \mathrm{GM}_d(\C^n) $ has nondegenerate tangent map for all $ d\ge 5 $ and $ n\ge 2 $. A proof of this would extend \Cref{thm:gm-nondefective}(3) to all degrees $ d\ge 5 $. 
\end{remark}

\subsubsection{Low-rank Gaussians} We can prove a similar result for rank-$ r $ Gaussians.  

\begin{cor}\label{cor:rank-r-gaussians}
	Let $ n, m, d, r\in \N $ and denote 
	\begin{align}
		\GM_{d, r}(\C^n) = \{s_{d}(\ell, q) \mid \rk q \le r \}.
	\end{align}
	Then, $ \GM_{d, r}(\C^n) $ is $ m $-nondefective, as long as 
	\begin{align}
		m \leq \frac{\dim S^d(\C^n)}{M_r} - M_r,  
	\end{align}
	where $ M_r :=  (r+1)n - \binom{r}{2} $. The generic rank is at most $ \frac{\dim S^d(\C^n)}{M_r} + M_r $. 
\end{cor}
\begin{proof}
	Counting parameters, we have that for $ d\ge 3 $, $ \GM_{d}(\mathcal{G}_{r}) $ has dimension equal to $ n $ (for the linear form) plus $ r(n-r) + \binom{r+1}{2} = rn - \binom{r}{2} $ for the rank-$ r $ covariance form. This sum equals $ M_r = (r+1) n - \binom{r}{2} $. Clearly, $ \GM_{d}(\mathcal{G}_{r}) $ is a $ \GL(\C^n) $-invariant, irreducible variety. The claim thus follows by \Cref{cor:generalized-nenashev-nondefectivity}. 
\end{proof}

\begin{remark}
	\Cref{cor:rank-r-gaussians} treats one example in the large category of \emph{Gaussian models}. These models are very common in applications and they arise from considering some subvariety $ \mathcal{G}\subseteq S^2(\C^n) $. The major motivation for a restriction to subvarieties is to reduce the number of parameters. This parameter reduction can have a significant impact on defectivity: Indeed, \Cref{cor:rank-r-gaussians} implies that, if the rank of the covariance matrices is an absolute constant, then a mixture of $ m = \mathcal{O}(n^2) $ general Gaussians is finite-to-one identifiable from its degree $ 3 $ moments; a sharp contrast to the case of full Gaussians, where moments of degree $ 5 $ are necessary. We believe that nondefectivity results for the rank-$ r $ Gaussian mixture model were not known before.
\end{remark}

\subsection{Other $G$-varieties}

In this section, we will see two applications of Theorem \ref{thm:introduction-main-result} for the case of other groups. Let $\C^{2n}$ be a $2n$-dimensional complex vector space and let $Q$ and $\omega$, respectively, a non-degenerate quadratic form and a non-degenerate skew symmetric form. 
Let $G=\SO(\C^{2n})$ be the  special orthogonal group with respect to $ Q $ and let $G=\Sp(\C^{2n})$ denote the symplectic group with respect to $ \omega $. We briefly recall the notations and the main result from \cite{Freire_Massarenti_Rischter_2020}.

\begin{defi}
	A subspace $ L \subseteq \C^{2n} $ of dimension $ n $ is called $ p $-\emph{isotropic} with respect to the nondegenerate bilinear form $ p $, if for all $ x\in L $, $ p(x, x) = 0 $. 
\end{defi}

We now introduce two special subvarieties of the Grassmannian that are invariants under the two groups introduced above.

\begin{defi}\label{lagrangian-spinor varieties}
	\noindent
	\begin{enumerate}[(a)]
		\item The \emph{Lagrangian Grassmannian} $ \mathcal{LG}(n,2n) \subset \grass(n,2n) $ is the variety of $n$-dimensional $\omega$-isotropic linear subspaces of $\C^{2n}$. It is $\Sp(\C^{2n})$-invariant and has dimension $\frac{n(n+1)}{2}$. 
		\item The \emph{Spinor variety} $\mathcal{S}_n \subset \grass(n,2n)$ is the variety of $n$-dimensional $Q$-isotropic linear subsapces of $\C^{2n}$. It is $\SO(\C^{2n})$-invariant of dimension $ \frac{n(n-1)}{2} $.
	\end{enumerate}
\end{defi}
%Note that $\mathcal{LG}(n,2n)$ comes naturally with a $\Sp(\C^{2n})$-action. 

The restriction of the Plücker map to $\mathcal{LG}(n,2n)$ embeds the Lagrangian Grassmannian in the irreducible $\Sp(\C^{2n})$-module $\C^{2n}_{\omega}$. The latter is defined as the complement of the submodule $\omega \wedge \bigwedge^{n-2} \C^{2n} \subset \bigwedge^{n}\C^{2n}$. It holds
\begin{align}
	\dim \C^{2n}_{\omega} = \frac{1}{2} \binom{2n}{n} + 2^{n-1}-1 % =\frac{1}{2}\sum_{j=1}^{n}({n \choose j}^2+{n \choose j}) =
\end{align}

Recall that $\mathcal{S}_n$ has two connected isomorphic components $\mathcal{S}_n^{+}$ and $\mathcal{S}_n^{-}$. Here, we implicitely make the choice $\mathcal{S}_n=\mathcal{S}_n^{+}$. Note that the restriction of the Plücker map embeds the Spinor variety into an irreducible $\SO(\C^{2n})$-module $\bigwedge^{n}\C^{2n}_{+}$, where $\bigwedge^{n}\C^{2n}_{+} \oplus \bigwedge^{n}\C^{2n}_{-}=\bigwedge^{n}\C^{2n}$ is the natural $\SO(\C^{2n})$-decomposition. 
The embedding $\mathcal{S}_n \to \bigwedge^{n}\C^{2n}_{+}$ is not the minimal one: Indeed, there is another embedding $\mathcal{S}_n \to \Delta$, where $\Delta$ is the half-spin $\SO(\C^{2n})$ representation. We have
\begin{align}
\dim \bigwedge^{n}\C^{2n}_{+}=\frac{1}{2}{2n \choose n} \text{ and } \dim \Delta=2^{n-1}	
\end{align}

The authors in \cite{Freire_Massarenti_Rischter_2020} proved that $\mathcal{LG}(n,2n) \subset \C^{2n}_{\omega}$ is not $m$-defective for $m $ at most $ \lfloor \frac{n+1}{2}   \rfloor$. Moreover, they showed that $\mathcal{S}_n \subset \bigwedge^{n}\C^{2n}_{+}$ is not $m$-defective for $m \leq \lfloor \frac{n}{2} \rfloor$ and that $\mathcal{S}_n \subset \Delta$ is not $m$-defective for $m \leq \lfloor \frac{n+2}{4} \rfloor$. 
Using Theorem \ref{thm:introduction-main-result}, we are able to significantly improve on the bounds of Massarenti and Rischter:

\begin{theorem}\label{thm:Lagrangian-Spinor}
The Lagrangian Grassmannian $\mathcal{LG}(n,2n)$ is $m$-nondefective (and $(m-1)$-identifiable) in the $\C^{2n}_{\omega}$ embedding, if $m$ is bounded by 
\begin{align}
\frac{1}{n(n+1)} \left(\binom{2n}{n} + 2^{n}-2\right) -\frac{n(n+1)}{2} %\sum_{j=1}^{n}\left({n \choose j}^2+{n \choose j}\right)
\end{align}
Furthermore, the generic rank $m^{\circ}$ is at most
\begin{align}
\frac{1}{n(n+1)} \left(\binom{2n}{n} + 2^{n}-2\right) +\frac{n(n+1)}{2}
\end{align}
Similarly, the Spinor variety $\mathcal{S}_n$ is $m$-nondefective (and $(m-1)$-identifiable) in the $\bigwedge^{n}\C^{2n}_{+}$ embedding, if $m$ is bounded by
\begin{align}
\frac{1}{n(n-1)}{2n \choose n}-\frac{n(n-1)}{2}
\end{align}
and in the $\Delta$ embedding, if $m$ is bounded by
\begin{align}
\frac{2^{n}}{n(n-1)}-\frac{n(n-1)}{2}
\end{align}
An analogous statement follows for the generic rank $m^{\circ}$.
\end{theorem}

Identifiability follows again from \cite[Theorem 1.5]{Massarenti_Mella_2022}, since both spinor varieties and Lagrangian Grassmannians are smooth. Galgano completely determined the locus of identifiability for the second secant of the spinor variety \cite{galgano2023identifiability}.

\subsection{Asymptotics and critical degree}

\begin{remark}\label{rem:asymptotic-optimality}
	The examples from \Cref{sec:intro} will be discussed thoroughly in \Cref{sec:applications}. However, note that the examples from the introduction have slightly more in common than just $ \GL $-invariance: Veronese varieties, powers-of-forms varieties, Chow varieties, Grassmann varieties and Gaussian moment varieties are images $ \varphi_d(L_d) $ (under equivariant maps) of a linear space $ L_{d} = L_{d}(\C^n) $ into spaces $ \mathcal{L}_d(\C^n) $ of $ d $-tensors on $ \C^n $.\footnote{In fact, $ L_{d}(\C^n) $ has the structure of a bounded-degree polynomial functor, for fixed $ d $. } In all of the above cases, the dimension of $ L_{d}(\C^n) $ is a polynomial $ p_{d}(n) $ in $ n $, whose degree $ d_0 := \deg p_d $ does not depend on $ d $. %NB: The Gaussian moment variety is only an embedding of $ \C^n \oplus S^2(\C^n) $, if $ d\ge 4 $. 
	
	Concretely, we have $ d_0 = 1 $ for the Veronese, Chow and Grassmannian varieties. We have $ d_0 = 2 $ for the Gaussian moment variety and $ d_0 = k $ for the variety of $ d $-th powers of $ k $-forms. We may now look at the bounds in \Cref{thm:introduction-main-result} and understand their asymptotical behaviour for $ n\to \infty $, just in terms of $ d_0 $. Indeed, if $ d > 2 d_0 $, then for large $ n $, almost all secants of $ \varphi_d(L_{d}(\C^n)) $ will be identifiable. The number of cases not covered by our theorem will be asymptotically negligible. The generic rank will be very close to the parameter count $ \frac{\dim \mathcal{L}_d(\C^n)}{\dim \varphi_d(L_{d}(\C^n)) } $. 
	
	On the other hand, for $ d < 2 d_0 $, our theorem becomes trivial (at least for large $ n $, but in practice often really fast). This is necessary, since in all of the examples above, for $ d\le d_0 $, already the second secant is defective. 
	Summarized, it can be said that all of the above examples have a ``\emph{critical degree}'' $ 2d_0 $, which separates the $ \varphi_{d} $ with mostly defective secants from those with mostly nondefective secants. 
\end{remark}

%\subsection{Other varieties} The list of invariant varieties may be continued indefinitely. 

%	\input{appendix.tex}
%	\input{remainders.tex}
%	\input{todo.tex}

	\bibliography{bibML}
	\bibliographystyle{plain}
\end{document}